\documentclass{article}
\usepackage{amsfonts}
\usepackage{amsmath}
\usepackage{harvard}

\newtheorem{theorem}{Theorem}

\newtheorem{definition}[theorem]{Definition}
\newtheorem{example}[theorem]{Example}

\newtheorem{remark}[theorem]{Remark}

\newenvironment{proof}[1][Proof]{\noindent\textbf{#1.} }{\ \rule{0.5em}{0.5em}}
\numberwithin{theorem}{section}
\numberwithin{equation}{section}
\input{tcilatex}

\begin{document}

\title{Multi-Time KCC-Invariants}
\author{Mircea Neagu \\
{\small 1 August 2009; Revised 7 August 2009 (minor corrections)}}
\date{}
\maketitle

\begin{abstract}
In this paper we construct some multi-time geometrical extensions of the
KCC-invariants, which characterize a given second-order system of PDEs on
the 1-jet space $J^{1}(T,M)$. A theorem of characterization of these
multi-time geometrical KCC-invariants is given.
\end{abstract}



\section{Geometrical objects on multi-time 1-jet spaces}

\hspace{5mm}We remind first few differential geometrical properties of the
multi-time 1-jet spaces. The multi-time 1-jet bundle%
\begin{equation*}
\xi _{1}=(J^{1}(T,M),\pi _{1},T\times M)
\end{equation*}%
is a vector bundle over the product manifold $T\times M$, having the fibre
type $\mathbb{R}^{mn}$, where $m$ (resp. $n$) is the dimension of the 
\textit{temporal} (resp. \textit{spatial}) manifold $T$ (resp. $M$). If the
temporal manifold $T$ has the local coordinates $(t^{\alpha })_{\alpha =%
\overline{1,m}}$ and the spatial manifold $M$ has the local coordinates $%
(x^{i})_{i=\overline{1,n}}$, then we denote the local coordinates of the
multi-time 1-jet space $J^{1}(T,M)$ by $(t^{\alpha },x^{i},x_{\alpha }^{i})$%
. These transform by the rules [8]%
\begin{equation}
\left\{ 
\begin{array}{l}
\widetilde{t}^{\alpha }=\widetilde{t}^{\alpha }(t^{\beta })\medskip \\ 
\widetilde{x}^{i}=\widetilde{x}^{i}(x^{j})\medskip \\ 
\widetilde{x}_{\alpha }^{i}={{\dfrac{\partial \widetilde{x}^{i}}{\partial
x^{j}}}{\dfrac{\partial t^{\beta }}{\partial \widetilde{t}^{\alpha }}}%
x_{\beta }^{j},}%
\end{array}%
\right.  \label{rgg}
\end{equation}%
where $\det (\partial \widetilde{t}^{\alpha }/\partial t^{\beta })\neq 0$
and $\det (\partial \widetilde{x}^{i}/\partial x^{j})\neq 0$.

\begin{remark}
In this work the greek indices $\alpha ,\beta ,\gamma ,\delta ,\varepsilon
,\mu ,\nu ,\rho ...$ run over the set $\left\{ 1,2,...,m\right\} $ and the
latin indices $i,j,k,l,p,q,r,s...$ run over the set $\left\{
1,2,...,n\right\} $. The Einstein convention of summation is also adopted
all over this paper.
\end{remark}

In the geometrical study of the multi-time 1-jet vector bundle, a central r%
\^{o}le is played by the \textit{distinguished tensors} ($d-$tensors).

\begin{definition}
A geometrical object $D=\left( D_{\gamma k(\beta )(l)\ldots }^{\alpha
i(j)(\nu )\ldots }\right) $ on the 1-jet vector bundle $J^{1}(T,M),$ whose
local components transform by the rules%
\begin{equation}
D_{\gamma k(\beta )(l)\ldots }^{\alpha i(j)(\nu )\ldots }=\widetilde{D}%
_{\varepsilon r(\mu )(s)\ldots }^{\delta p(q)(\eta )\ldots }{\frac{\partial
t^{\alpha }}{\partial \widetilde{t}^{\delta }}}{\frac{\partial x^{i}}{%
\partial \widetilde{x}^{p}}}\left( {\frac{\partial x^{j}}{\partial 
\widetilde{x}^{q}}}{\frac{\partial \widetilde{t}^{\mu }}{\partial t^{\beta }}%
}\right) {\frac{\partial \widetilde{t}^{\varepsilon }}{\partial t^{\gamma }}}%
{\frac{\partial \widetilde{x}^{r}}{\partial x^{k}}}\left( {\frac{\partial 
\widetilde{x}^{s}}{\partial x^{l}}}{\frac{\partial t^{\nu }}{\partial 
\widetilde{t}^{\eta }}}\right) \ldots \;,  \label{tr-rules-$d-$tensors}
\end{equation}%
is called a \textit{$d-$}\textbf{tensor field}.
\end{definition}

\begin{remark}
The use of parentheses for certain indices of the local components $%
D_{\gamma k(\beta )(l)\ldots }^{\alpha i(j)(\nu )\ldots }$ of the
distinguished tensor field $D$ on the 1-jet space is motivated by the fact
that the pair of indices $"$ $_{(\beta )}^{(j)}$ $"$ or $"$ $_{(l)}^{(\nu )}$
$"$ behaves like a single index.
\end{remark}

\begin{example}
\label{Liouville} The geometrical object%
\begin{equation*}
\mathbf{C}=\mathbf{C}_{(\alpha )}^{(i)}\dfrac{\partial }{\partial x_{\alpha
}^{i}},
\end{equation*}%
where $\mathbf{C}_{(\alpha )}^{(i)}=x_{\alpha }^{i},$ represents a $d-$%
tensor field on the 1-jet space; this is called the \textbf{canonical
Liouville }$d-$\textbf{tensor field} of the 1-jet vector bundle $J^{1}(T,M)$
and it is a global geometrical object.
\end{example}

\begin{example}
\label{normal}Let $h=(h_{\alpha \beta }(t))$ be a Riemannian metric on the
temporal ma\-ni\-fold $T$. The geometrical object 
\begin{equation*}
\mathbf{J}_{h}=J_{(\alpha )\beta j}^{(i)}\dfrac{\partial }{\partial
x_{\alpha }^{i}}\otimes dt^{\beta }\otimes dx^{j},
\end{equation*}%
where $J_{(\alpha )\beta j}^{(i)}=h_{\alpha \beta }\delta _{j}^{i}$ is a $d-$%
tensor field on $J^{1}(T,M),$ which is called the $h$\textbf{-normalization }%
$d-$\textbf{tensor field} of the 1-jet space $J^{1}(T,M)$. Obviously, it is
also a global geometrical object.
\end{example}

In the Riemann-Lagrange differential geometry of the 1-jet spaces developed
in [8], [9] important r\^{o}les are also played by geometrical objects as
the \textit{temporal} or \textit{spatial semisprays}, together with the 
\textit{multi-time jet nonlinear connections}.

\begin{definition}
A set of local functions $H=\left( H_{(\alpha )\beta }^{(i)}\right) $ on $%
J^{1}(T,M),$ which transform by the rules%
\begin{equation}
2\widetilde{H}_{(\alpha )\beta }^{(i)}=2H_{(\gamma )\nu }^{(k)}{\frac{%
\partial \widetilde{x}^{i}}{\partial x^{k}}\frac{\partial t^{\gamma }}{%
\partial \widetilde{t}^{\alpha }}\frac{\partial t^{\nu }}{\partial 
\widetilde{t}^{\beta }}}-{\frac{\partial t^{\mu }}{\partial \widetilde{t}%
^{\beta }}}{\frac{\partial \widetilde{x}_{\alpha }^{i}}{\partial t^{\mu }}},
\label{tr-rules-t-s}
\end{equation}%
is called a \textbf{temporal semispray} on $J^{1}(T,M)$.
\end{definition}

\begin{example}
\label{H0} Let us consider a Riemannian metric $h=(h_{\alpha \beta }(t))$ on
the temporal manifold $T$ and let%
\begin{equation*}
H_{\beta \gamma }^{\alpha }=\frac{h^{\alpha \mu }}{2}\left( \frac{\partial
h_{\beta \mu }}{\partial t^{\gamma }}+\frac{\partial h_{\gamma \mu }}{%
\partial t^{\beta }}-\frac{\partial h_{\beta \gamma }}{\partial t^{\mu }}%
\right)
\end{equation*}%
be its Christoffel symbols. Taking into account that we have the
transformation rules%
\begin{equation}
\widetilde{H}_{\nu \rho }^{\delta }=H_{\beta \gamma }^{\alpha }\frac{%
\partial \widetilde{t}^{\delta }}{\partial t^{\alpha }}\frac{\partial
t^{\beta }}{\partial \widetilde{t}^{\nu }}\frac{\partial t^{\gamma }}{%
\partial \widetilde{t}^{\rho }}+\frac{\partial \widetilde{t}^{\delta }}{%
\partial t^{\varepsilon }}\frac{\partial ^{2}t^{\varepsilon }}{\partial 
\widetilde{t}^{\nu }\partial \widetilde{t}^{\rho }},  \label{t-Cris-symb}
\end{equation}%
we deduce that the local components%
\begin{equation*}
\mathring{H}_{(\alpha )\beta }^{(i)}=-\frac{1}{2}H_{\alpha \beta }^{\mu
}x_{\mu }^{i}
\end{equation*}%
define a temporal semispray $\mathring{H}=\left( \mathring{H}_{(\alpha
)\beta }^{(i)}\right) $ on $J^{1}(T,M)$. This is called the \textbf{%
canonical temporal semispray associated to the temporal metric }$h_{\alpha
\beta }(t)$.
\end{example}

\begin{definition}
A set of local functions $G=\left( G_{(\alpha )\beta }^{(i)}\right) ,$ which
transform by the rules%
\begin{equation}
2\widetilde{G}_{(\alpha )\beta }^{(i)}=2G_{(\gamma )\nu }^{(k)}{\frac{%
\partial \widetilde{x}^{i}}{\partial x^{k}}}{\frac{\partial t^{\gamma }}{%
\partial \widetilde{t}^{\alpha }}}\frac{\partial t^{\nu }}{\partial 
\widetilde{t}^{\beta }}-{\frac{\partial x^{r}}{\partial \widetilde{x}^{s}}}{%
\frac{\partial \widetilde{x}_{\alpha }^{i}}{\partial x^{r}}}\widetilde{x}%
_{\beta }^{s},  \label{tr-rules-s-s}
\end{equation}%
is called a \textbf{spatial semispray} on $J^{1}(T,M)$.
\end{definition}

\begin{example}
\label{G0} Let $\varphi =(\varphi _{ij}(x))$ be a Riemannian metric on the
spatial manifold $M$ and let us consider%
\begin{equation*}
\gamma _{jk}^{i}=\frac{\varphi ^{im}}{2}\left( \frac{\partial \varphi _{jm}}{%
\partial x^{k}}+\frac{\partial \varphi _{km}}{\partial x^{j}}-\frac{\partial
\varphi _{jk}}{\partial x^{m}}\right)
\end{equation*}%
its Christoffel symbols. Taking into account that we have the transformation
rules%
\begin{equation}
\widetilde{\gamma }_{qr}^{p}=\gamma _{jk}^{i}\frac{\partial \widetilde{x}^{p}%
}{\partial x^{i}}\frac{\partial x^{j}}{\partial \widetilde{x}^{q}}\frac{%
\partial x^{k}}{\partial \widetilde{x}^{r}}+\frac{\partial \widetilde{x}^{p}%
}{\partial x^{l}}\frac{\partial ^{2}x^{l}}{\partial \widetilde{x}%
^{q}\partial \widetilde{x}^{r}},  \label{s-Cris-symb}
\end{equation}%
we deduce that the local components%
\begin{equation*}
\mathring{G}_{(\alpha )\beta }^{(i)}=\frac{1}{2}\gamma _{pq}^{i}x_{\alpha
}^{p}x_{\beta }^{q}
\end{equation*}%
define a spatial semispray $\mathring{G}=\left( \mathring{G}_{(\alpha )\beta
}^{(i)}\right) $ on $J^{1}(T,M)$. This is called the \textbf{canonical
spatial semispray associated to the spatial metric }$\varphi _{ij}(x)$.
\end{example}

\begin{definition}
A set of local functions $\Gamma =\left( M_{(\alpha )\beta
}^{(i)},N_{(\alpha )j}^{(i)}\right) $ on $J^{1}(T,M),$ which transform by
the rules%
\begin{equation}
\widetilde{M}{_{(\alpha )\beta }^{(i)}=M_{(\gamma )\nu }^{(k)}{\dfrac{%
\partial \widetilde{x}^{i}}{\partial x^{k}}}{\dfrac{\partial t^{\gamma }}{%
\partial \widetilde{t}^{\alpha }}}\frac{\partial t^{\nu }}{\partial 
\widetilde{t}^{\beta }}-\frac{\partial t^{\mu }}{\partial \widetilde{t}%
^{\beta }}{\dfrac{\partial \widetilde{x}_{\alpha }^{i}}{\partial t^{\mu }}}}
\label{tr-rules-t-nlc}
\end{equation}%
and%
\begin{equation}
\widetilde{N}{_{(\alpha )j}^{(i)}=N_{(\gamma )l}^{(k)}{\dfrac{\partial 
\widetilde{x}^{i}}{\partial x^{k}}}{\dfrac{\partial t^{\gamma }}{\partial 
\widetilde{t}^{\alpha }}}\frac{\partial x^{l}}{\partial \widetilde{x}^{j}}-%
\frac{\partial x^{r}}{\partial \widetilde{x}^{j}}{\dfrac{\partial \widetilde{%
x}_{\alpha }^{i}}{\partial x^{r}}}},  \label{tr-rules-s-nlc}
\end{equation}%
is called a \textbf{nonlinear connection} on the 1-jet space $J^{1}(T,M)$.
\end{definition}

\begin{example}
Let us consider that $(T,h_{\alpha \beta }(t))$ and $(M,\varphi _{ij}(x))$
are Rie\-ma\-nni\-an manifolds having the Christoffel symbols $H_{\beta
\gamma }^{\alpha }(t)$ and $\gamma _{jk}^{i}(x)$. Then, using the
transformation rules (\ref{rgg}), (\ref{t-Cris-symb}) and (\ref{s-Cris-symb}%
), we deduce that the set of local functions%
\begin{equation*}
\mathring{\Gamma}=\left( \mathring{M}_{(\alpha )\beta }^{(i)},\mathring{N}%
_{(\alpha )j}^{(i)}\right) ,
\end{equation*}%
where%
\begin{equation*}
\mathring{M}_{(\alpha )\beta }^{(i)}=-H_{\alpha \beta }^{\mu }x_{\mu }^{i}%
\text{ \ \ and \ \ }\mathring{N}_{(\alpha )j}^{(i)}=\gamma
_{jr}^{i}x_{\alpha }^{r},
\end{equation*}%
represents a nonlinear connection on the 1-jet space $J^{1}(T,M)$. This
multi-time jet nonlinear connection is called the \textbf{canonical
nonlinear connection attached to the pair of Riemannian metrics }$(h_{\alpha
\beta }(t),\varphi _{ij}(x))$.
\end{example}

In the sequel, let us study the geometrical relations between \textit{%
temporal }or\textit{\ spatial semisprays} and \textit{multi-time nonlinear
connections} on the 1-jet vector bundle $J^{1}(T,M)$. In this direction,
using the local transformation laws (\ref{tr-rules-t-s}), (\ref%
{tr-rules-t-nlc}) and (\ref{rgg}), respectively the transformation laws (\ref%
{tr-rules-s-s}), (\ref{tr-rules-s-nlc}) and (\ref{rgg}), by direct local
computations, we find the following geometrical results:

\begin{theorem}
The \textit{temporal semisprays} $H=\left( H_{(\alpha )\beta }^{(i)}\right) $
and the sets of \textit{temporal components of nonlinear connections }$%
\Gamma _{\text{temporal}}=\left( M_{(\alpha )\beta }^{(i)}\right) $ are in
one-to-one correspondence on the 1-jet space $J^{1}(T,M),$ via: 
\begin{equation*}
M_{(\alpha )\beta }^{(i)}=2H_{(\alpha )\beta }^{(i)},\qquad {H_{(\alpha
)\beta }^{(i)}={\frac{1}{2}}M_{(\alpha )\beta }^{(i)}};
\end{equation*}
\end{theorem}

\begin{theorem}
\textbf{(i)} If $G_{(\alpha )\beta }^{(i)}$ are the components of a spatial
semispray on $J^{1}(T,M),$ where $(T,h_{\alpha \beta }(t))$ is a Riemannian
manifold$,$ then the components%
\begin{equation*}
N_{(\alpha )j}^{(i)}={\frac{\partial G^{i}}{\partial x_{\mu }^{j}}}h_{\alpha
\mu },
\end{equation*}%
where $G^{i}=h^{\delta \varepsilon }G_{(\delta )\varepsilon }^{(i)},$
represent a spatial nonlinear connection on $J^{1}(T,M)$.

\textbf{(ii)} Conversely$,$ the spatial nonlinear connection $\Gamma _{\text{%
spatial}}=\left( N_{(\alpha )j}^{(i)}\right) $ produces the spatial
semispray components%
\begin{equation*}
G_{(\alpha )\beta }^{(i)}=\frac{1}{2}N_{(\alpha )r}^{(i)}x_{\beta }^{r}.
\end{equation*}
\end{theorem}

\section{Multi-time geometrical KCC-theory}

\hspace{5mm}In this Section we construct some multi-time generalizations on
the 1-jet space $J^{1}(T,M)$ for the basic objects of the KCC-theory ([1],
[2], [3], [10]). In this respect, let us consider on $J^{1}(T,M)$ a
second-order system of partial differential equations of local form%
\begin{equation}
\frac{\partial ^{2}x^{i}}{\partial t^{\alpha }\partial t^{\beta }}%
+F_{(\alpha )\beta }^{(i)}(t^{\gamma },x^{k},x_{\gamma }^{k})=0,\quad \alpha
,\beta =\overline{1,m},\quad i=\overline{1,n},  \label{SODE}
\end{equation}%
where $x_{\gamma }^{k}=\partial x^{k}/\partial t^{\gamma }$, $F_{(\alpha
)\beta }^{(i)}=F_{(\beta )\alpha }^{(i)}$ and the local components $%
F_{(\alpha )\beta }^{(i)}(t^{\gamma },x^{k},x_{\gamma }^{k})$ transform
under a change of coordinates (\ref{rgg}) by the rules%
\begin{equation}
\widetilde{F}_{(\alpha )\beta }^{(i)}=2F_{(\gamma )\nu }^{(k)}{\frac{%
\partial \widetilde{x}^{i}}{\partial x^{k}}\frac{\partial t^{\gamma }}{%
\partial \widetilde{t}^{\alpha }}\frac{\partial t^{\nu }}{\partial 
\widetilde{t}^{\beta }}}-{\frac{\partial t^{\mu }}{\partial \widetilde{t}%
^{\beta }}}{\frac{\partial \widetilde{x}_{\alpha }^{i}}{\partial t^{\mu }}}-{%
\frac{\partial x^{r}}{\partial \widetilde{x}^{s}}}{\frac{\partial \widetilde{%
x}_{\alpha }^{i}}{\partial x^{r}}}\widetilde{x}_{\beta }^{s},.
\label{transformations-F}
\end{equation}

\begin{remark}
The second-order system of partial differential equations (\ref{SODE}) is
invariant under a change of coordinates (\ref{rgg}).
\end{remark}

\begin{example}
Let us consider that $(T,h_{\alpha \beta }(t))$ and $(M,\varphi _{ij}(x))$
are Rie\-ma\-nni\-an manifolds having the Christoffel symbols $H_{\beta
\gamma }^{\alpha }(t)$ and $\gamma _{jk}^{i}(x)$. Then, the local components%
\begin{equation*}
\mathring{F}_{(\alpha )\beta }^{(i)}=-H_{\alpha \beta }^{\mu }x_{\mu
}^{i}+\gamma _{pq}^{i}x_{\alpha }^{p}x_{\beta }^{q}
\end{equation*}%
transform under a change of coordinates (\ref{rgg}) by the rules (\ref%
{transformations-F}). In this particular case, the PDE system (\ref{SODE})
becomes%
\begin{equation}
\frac{\partial ^{2}x^{i}}{\partial t^{\alpha }\partial t^{\beta }}-H_{\alpha
\beta }^{\mu }x_{\mu }^{i}+\gamma _{pq}^{i}x_{\alpha }^{p}x_{\beta
}^{q}=0,\quad \alpha ,\beta =\overline{1,m},\quad i=\overline{1,n},
\label{affine-eq}
\end{equation}%
that is the PDE system of the affine maps between the Riemannian manifolds $%
(T,h_{\alpha \beta }(t))$ and $(M,\varphi _{ij}(x))$. We recall that these
affine maps carry the geodesics of the temporal Riemannian manifold $%
(T,h_{\alpha \beta }(t))$ into the geodesics of the spatial Riemannian
manifold $(M,\varphi _{ij}(x))$. Moreover, the $h-$trace of the equations (%
\ref{affine-eq}) produces the equations of the harmonic maps between the
Riemannian manifolds $(T,h_{\alpha \beta }(t))$ and $(M,\varphi _{ij}(x))$.
For more details, please see [6].
\end{example}

Using a temporal Riemannian metric $h_{\alpha \beta }(t)$ on $T$ and taking
into account the transformation rules (\ref{tr-rules-t-s}), (\ref%
{tr-rules-s-s}) and (\ref{transformations-F}), we can rewrite the PDE system
(\ref{SODE}) in the following form:%
\begin{equation*}
\frac{\partial ^{2}x^{i}}{\partial t^{\alpha }\partial t^{\beta }}-H_{\alpha
\beta }^{\mu }x_{\mu }^{i}+2G_{(\alpha )\beta }^{(i)}(t^{\gamma
},x^{k},x_{\gamma }^{k})=0,\quad \alpha ,\beta =\overline{1,m},\quad i=%
\overline{1,n},
\end{equation*}%
where%
\begin{equation*}
G_{(\alpha )\beta }^{(i)}=\frac{1}{2}F_{(\alpha )\beta }^{(i)}+\frac{1}{2}%
H_{\alpha \beta }^{\mu }x_{\mu }^{i}
\end{equation*}%
are the components of a spatial semispray on $J^{1}(T,M)$. The coefficients
of the spatial semispray $G_{(\alpha )\beta }^{(i)}$ produce the spatial
components $N_{(\alpha )j}^{(i)}$ of a nonlinear connection $\Gamma $ on the
1-jet space $J^{1}(T,M)$, by putting%
\begin{equation*}
N_{(\alpha )j}^{(i)}=\frac{h^{\mu \nu }\partial G_{(\mu )\nu }^{(i)}}{%
\partial x_{\gamma }^{j}}h_{\gamma \alpha }=\frac{h^{\mu \nu }}{2}\frac{%
\partial F_{(\mu )\nu }^{(i)}}{\partial x_{\gamma }^{j}}h_{\gamma \alpha }+%
\frac{h^{\mu \nu }}{2}H_{\mu \nu }^{\gamma }h_{\gamma \alpha }\delta
_{j}^{i}.
\end{equation*}

In order to find the basic jet multi-time geometrical invariants of the PDE
system (\ref{SODE}) (please see Kosambi [7], Cartan [4] and Chern [5]) under
the coordinate transformations (\ref{rgg}), we define the $h-$\textit{%
KCC-covariant derivative of a $d-$tensor of type }$T_{(\alpha
)}^{(i)}(t^{\gamma },x^{k},x_{\gamma }^{k})$ on the 1-jet space $J^{1}(T,M)$%
, via%
\begin{eqnarray*}
\frac{\overset{h}{\nabla }T_{(\alpha )}^{(i)}}{\partial t^{\beta }} &=&\frac{%
\partial T_{(\alpha )}^{(i)}}{\partial t^{\beta }}+N_{(\alpha
)r}^{(i)}T_{(\beta )}^{(r)}-H_{\alpha \beta }^{\mu }T_{(\mu )}^{(i)}= \\
&=&\frac{\partial T_{(\alpha )}^{(i)}}{\partial t^{\beta }}+\frac{h^{\mu \nu
}}{2}\frac{\partial F_{(\mu )\nu }^{(i)}}{\partial x_{\gamma }^{r}}h_{\gamma
\alpha }T_{(\beta )}^{(r)}+\frac{h^{\mu \nu }}{2}H_{\mu \nu }^{\gamma
}h_{\gamma \alpha }T_{(\beta )}^{(i)}-H_{\alpha \beta }^{\mu }T_{(\mu
)}^{(i)}.
\end{eqnarray*}

\begin{remark}
The $h-$\textit{KCC-covariant derivative} components $\dfrac{\overset{h}{%
\nabla }T_{(\alpha )}^{(i)}}{\partial t^{\beta }}$ transform under a change
of coordinates (\ref{rgg}) as a $d-$tensor of type $\mathcal{T}_{(\alpha
)\beta }^{(i)}.$
\end{remark}

In such a geometrical context, if we use the notation $x_{\alpha
}^{i}=\partial x^{i}/\partial t^{\alpha }$, then the PDE system (\ref{SODE})
can be rewritten in the following distinguished tensorial form:%
\begin{eqnarray*}
\frac{\overset{h}{\nabla }x_{\alpha }^{i}}{\partial t^{\beta }}
&=&-F_{(\alpha )\beta }^{(i)}(t^{\gamma },x^{k},x_{\gamma }^{k})+N_{(\alpha
)r}^{(i)}x_{\beta }^{r}-H_{\alpha \beta }^{\mu }x_{\mu }^{i}= \\
&=&-F_{(\alpha )\beta }^{(i)}+\frac{h^{\mu \nu }}{2}\frac{\partial F_{(\mu
)\nu }^{(i)}}{\partial x_{\gamma }^{r}}h_{\gamma \alpha }x_{\beta }^{r}+%
\frac{h^{\mu \nu }}{2}H_{\mu \nu }^{\gamma }h_{\gamma \alpha }x_{\beta
}^{i}-H_{\alpha \beta }^{\mu }x_{\mu }^{i}.
\end{eqnarray*}

\begin{definition}
The distinguished tensor%
\begin{equation*}
\overset{h}{\varepsilon }\text{ }\!\!_{(\alpha )\beta }^{(i)}=-F_{(\alpha
)\beta }^{(i)}+\frac{h^{\mu \nu }}{2}\frac{\partial F_{(\mu )\nu }^{(i)}}{%
\partial x_{\gamma }^{r}}h_{\gamma \alpha }x_{\beta }^{r}+\frac{h^{\mu \nu }%
}{2}H_{\mu \nu }^{\gamma }h_{\gamma \alpha }x_{\beta }^{i}-H_{\alpha \beta
}^{\mu }x_{\mu }^{i}
\end{equation*}%
is called the \textbf{first multi-time }$h-$\textbf{KCC-invariant} on the
1-jet space $J^{1}(T,M)$ of the PDEs (\ref{SODE}), which can be interpreted
as an \textbf{external force} [1], [3].
\end{definition}

\begin{example}
For the second-order PDE system (\ref{affine-eq})$,$ which gives the affine
maps between the Riemannian manifolds $(T,h_{\alpha \beta }(t))$ and $%
(M,\varphi _{ij}(x)),$ the first multi-time $h-$KCC-invariant is zero.
\end{example}

\begin{example}
It can be easily seen that for the particular first order PDE system%
\begin{equation}
\frac{\partial x^{i}}{\partial t^{\alpha }}=X_{(\alpha )}^{(i)}(t^{\gamma
},x^{k})\Rightarrow \frac{\partial ^{2}x^{i}}{\partial t^{\alpha }\partial
t^{\beta }}=\frac{\partial X_{(\alpha )}^{(i)}}{\partial t^{\beta }}+\frac{%
\partial X_{(\alpha )}^{(i)}}{\partial x^{r}}x_{\beta }^{r},  \label{Jet_DS}
\end{equation}%
where $X_{(\alpha )}^{(i)}(t,x)$ is a given $d-$tensor on $J^{1}(T,M),$ the
first multi-time $h-$KCC-invariant has the form%
\begin{equation*}
\overset{h}{\varepsilon }\text{ }\!\!_{(\alpha )\beta }^{(i)}=\frac{\partial
X_{(\alpha )}^{(i)}}{\partial t^{\beta }}+\frac{1}{2}\frac{\partial
X_{(\alpha )}^{(i)}}{\partial x^{r}}x_{\beta }^{r}+\frac{h^{\mu \nu }}{2}%
H_{\mu \nu }^{\gamma }h_{\gamma \alpha }x_{\beta }^{i}-H_{\alpha \beta
}^{\mu }x_{\mu }^{i}.
\end{equation*}
\end{example}

In the sequel, let us vary the solutions $x^{i}(t^{\gamma })$ of the PDE
system (\ref{SODE}) by the nearby smooth maps $(\overline{x}^{i}(t^{\gamma
},s))_{s\in (-\varepsilon ,\varepsilon )},$ where $\overline{x}%
^{i}(t^{\gamma },0)=x^{i}(t^{\gamma }).$ Then, if we consider the \textit{%
variation $d-$tensor field}%
\begin{equation*}
\xi ^{i}(t^{\gamma })=\left. \dfrac{\partial \overline{x}^{i}}{\partial s}%
\right\vert _{s=0},
\end{equation*}%
we get the \textit{variational equations}%
\begin{equation}
\frac{\partial ^{2}\xi ^{i}}{\partial t^{\alpha }\partial t^{\beta }}+\frac{%
\partial F_{(\alpha )\beta }^{(i)}}{\partial x^{k}}\xi ^{k}+\frac{\partial
F_{(\alpha )\beta }^{(i)}}{\partial x_{\mu }^{r}}\frac{\partial \xi ^{r}}{%
\partial t^{\mu }}=0,\quad \alpha ,\beta =\overline{1,m},\quad i=\overline{%
1,n}.  \label{Var-Equations}
\end{equation}%
It is obvious that the equations (\ref{Var-Equations}) imply the $h-$\textit{%
trace variational equations}%
\begin{equation}
h^{\alpha \beta }\frac{\partial ^{2}\xi ^{i}}{\partial t^{\alpha }\partial
t^{\beta }}+\frac{\partial F^{i}}{\partial x^{k}}\xi ^{k}+\frac{\partial
F^{i}}{\partial x_{\mu }^{r}}\frac{\partial \xi ^{r}}{\partial t^{\mu }}%
=0,\quad i=\overline{1,n},  \label{h-Trace-Var-Eq}
\end{equation}%
where $F^{i}=h^{\alpha \beta }F_{(\alpha )\beta }^{(i)}$.

To find other multi-time geometrical invariants for the PDE system (\ref%
{SODE}), we also introduce the $h-$\textit{KCC-covariant derivative of a $d-$%
tensor of type }$\xi ^{i}(t^{\gamma })$ on the 1-jet space $J^{1}(T,M)$, via 
\begin{equation*}
\frac{\overset{h}{\nabla }\xi ^{i}}{\partial t^{\alpha }}=\frac{\partial \xi
^{i}}{\partial t^{\alpha }}+N_{(\alpha )r}^{(i)}\xi ^{r}=\frac{\partial \xi
^{i}}{\partial t^{\alpha }}+\frac{1}{2}\frac{\partial F^{i}}{\partial
x_{\gamma }^{r}}h_{\gamma \alpha }\xi ^{r}+\frac{1}{2}H^{\gamma }h_{\gamma
\alpha }\xi ^{i},
\end{equation*}%
where $H^{\gamma }=h^{\mu \nu }H_{\mu \nu }^{\gamma }.$

\begin{remark}
The $h-$\textit{KCC-covariant derivative} components $\dfrac{\overset{h}{%
\nabla }\xi ^{i}}{\partial t^{\alpha }}$ transform under a change of
coordinates (\ref{rgg}) as a $d-$tensor of type $\mathrm{T}_{(\alpha
)}^{(i)}.$
\end{remark}

In this geometrical context, the $h-$trace variational equations (\ref%
{h-Trace-Var-Eq}) can be rewritten in the following distinguished tensorial
form:%
\begin{equation*}
h^{\alpha \beta }\frac{\overset{h}{\nabla }}{\partial t^{\beta }}\left[ 
\frac{\overset{h}{\nabla }\xi ^{i}}{\partial t^{\alpha }}\right] =\overset{h}%
{P}\text{ \negthinspace \negthinspace }_{r}^{i}\xi ^{r},
\end{equation*}%
where%
\begin{equation*}
\begin{array}{lll}
\overset{h}{P}\text{ \negthinspace \negthinspace }_{j}^{i} & = & -\dfrac{%
\partial F^{i}}{\partial x^{j}}+\dfrac{1}{2}\dfrac{\partial ^{2}F^{i}}{%
\partial t^{\gamma }\partial x_{\gamma }^{j}}+\dfrac{1}{2}\dfrac{\partial
^{2}F^{i}}{\partial x^{r}\partial x_{\gamma }^{j}}x_{\gamma }^{r}-\dfrac{1}{2%
}\dfrac{\partial ^{2}F^{i}}{\partial x_{\mu }^{j}\partial x_{\gamma }^{r}}%
F_{(\gamma )\mu }^{(r)}+\medskip \\ 
&  & +\dfrac{h_{\gamma \mu }}{4}\dfrac{\partial F^{i}}{\partial x_{\gamma
}^{r}}\dfrac{\partial F^{r}}{\partial x_{\mu }^{j}}+\dfrac{h^{\gamma \eta }}{%
2}\dfrac{\partial h\mu \gamma }{\partial t^{\eta }}\dfrac{\partial F^{i}}{%
\partial x_{\mu }^{j}}+\medskip \\ 
&  & +\left[ \dfrac{1}{2}\dfrac{\partial H^{\gamma }}{\partial t^{\gamma }}+%
\dfrac{h^{\gamma \eta }}{2}\dfrac{\partial h\mu \gamma }{\partial t^{\eta }}%
H^{\mu }-\dfrac{h_{\gamma \mu }}{4}H^{\gamma }H^{\mu }\right] \delta
_{j}^{i}.%
\end{array}%
\end{equation*}

\begin{definition}
The $d-$tensor $\overset{h}{P}$ \negthinspace $_{j}^{i}$ is called the 
\textbf{second multi-time }$h-$\textbf{KCC-invariant} on the 1-jet space $%
J^{1}(T,M)$ of the PDE system (\ref{SODE}), or the \textbf{multi-time }$h-$%
\textbf{deviation curvature }$d-$\textbf{tensor}.
\end{definition}

\begin{example}
If we consider the second-order PDE system of the affine maps between the
Riemannian manifolds $(T,h_{\alpha \beta }(t))$ and $(M,\varphi _{ij}(x)),$
system which is given by (\ref{affine-eq})$,$ then the second multi-time $h-$%
KCC-invariant has the form%
\begin{equation*}
\overset{h}{P}\text{ \negthinspace \negthinspace }_{j}^{i}=-h^{\alpha \beta }%
\mathfrak{R}_{pqj}^{i}x_{\alpha }^{p}x_{\beta }^{q},
\end{equation*}%
where%
\begin{equation*}
\mathfrak{R}_{pqj}^{i}=\frac{\partial \gamma _{pq}^{i}}{\partial x^{j}}-%
\frac{\partial \gamma _{pj}^{i}}{\partial x^{q}}+\gamma _{pq}^{r}\gamma
_{rj}^{i}-\gamma _{pj}^{r}\gamma _{rq}^{i}
\end{equation*}%
are the components of the curvature of the spatial Riemannian metric $%
\varphi _{ij}(x).$ Consequently, the $h-$trace variational equations (\ref%
{h-Trace-Var-Eq}) become the following \textbf{mul\-ti\--time }$h-$\textbf{%
Jacobi field equations:}%
\begin{equation*}
h^{\alpha \beta }\left\{ \frac{\overset{h}{\nabla }}{\partial t^{\beta }}%
\left[ \frac{\overset{h}{\nabla }\xi ^{i}}{\partial t^{\alpha }}\right] +%
\mathfrak{R}_{pqr}^{i}x_{\alpha }^{p}x_{\beta }^{q}\xi ^{r}\right\} =0,
\end{equation*}%
where%
\begin{equation*}
\frac{\overset{h}{\nabla }\xi ^{i}}{\partial t^{\alpha }}=\frac{\partial \xi
^{i}}{\partial t^{\alpha }}+\gamma _{pr}^{i}x_{\alpha }^{p}\xi ^{r}.
\end{equation*}
\end{example}

\begin{example}
For the particular first order PDE system (\ref{Jet_DS}) the multi-time $h-$%
deviation curvature $d-$tensor is given by%
\begin{equation*}
\begin{array}{lll}
\overset{h}{P}\text{ \negthinspace \negthinspace }_{j}^{i} & = & \dfrac{%
h^{\alpha \beta }}{2}\left[ \dfrac{\partial ^{2}X_{(\alpha )}^{(i)}}{%
\partial t^{\beta }\partial x^{j}}+\dfrac{\partial ^{2}X_{(\alpha )}^{(i)}}{%
\partial x^{j}\partial x^{r}}x_{\beta }^{r}+\dfrac{1}{2}\dfrac{\partial
X_{(\alpha )}^{(i)}}{\partial x^{r}}\dfrac{\partial X_{(\beta )}^{(r)}}{%
\partial x^{j}}\right] +\medskip \\ 
&  & +\left[ \dfrac{1}{2}\dfrac{\partial H^{\gamma }}{\partial t^{\gamma }}+%
\dfrac{h^{\gamma \eta }}{2}\dfrac{\partial h\mu \gamma }{\partial t^{\eta }}%
H^{\mu }-\dfrac{h_{\gamma \mu }}{4}H^{\gamma }H^{\mu }\right] \delta
_{j}^{i}.%
\end{array}%
\end{equation*}
\end{example}

\begin{definition}
The distinguished tensors%
\begin{equation*}
\overset{h}{R}\text{ \negthinspace \negthinspace }_{jk}^{i\alpha }=\frac{1}{3%
}\left[ \frac{\partial \overset{h}{P}\text{ \negthinspace \negthinspace }%
_{j}^{i}}{\partial x_{\alpha }^{k}}-\frac{\partial \overset{h}{P}\text{
\negthinspace \negthinspace }_{k}^{i}}{\partial x_{\alpha }^{j}}\right]
,\qquad \overset{h}{B}\text{ \negthinspace \negthinspace }_{jk(l)}^{i\alpha
(\beta )}=\frac{\partial \overset{h}{R}\text{ \negthinspace \negthinspace }%
_{jk}^{i\alpha }}{\partial x_{\beta }^{l}}
\end{equation*}%
and\ 
\begin{equation*}
D_{(\alpha )\beta (j)(k)(l)}^{(i)\text{ \ }(\gamma )(\varepsilon )(\mu )}=%
\frac{\partial ^{3}F_{(\alpha )\beta }^{(i)}}{\partial x_{\gamma
}^{j}\partial x_{\varepsilon }^{k}\partial x_{\mu }^{l}}
\end{equation*}%
are called the \textbf{third}, \textbf{fourth}\textit{\ }and\textit{\ }%
\textbf{fifth multi-time }$h-$\textbf{KCC-invariant}\textit{\ }on the 1-jet
vector bundle $J^{1}(T,M)$ of the PDE system (\ref{SODE}).
\end{definition}

\begin{remark}
Taking into account the transformation rules (\ref{transformations-F}) of
the components $F_{(\alpha )\beta }^{(i)}$, we immediately deduce that the
components $D_{(\alpha )\beta (j)(k)(l)}^{(i)\text{ \ }(\gamma )(\varepsilon
)(\mu )}$ behave like a $d-$tensor on the 1-jet space $J^{1}(T,M)$.
\end{remark}

\begin{example}
For the second-order PDE system (\ref{affine-eq}) of the affine maps between
the Riemannian manifolds $(T,h_{\alpha \beta }(t))$ and $(M,\varphi
_{ij}(x)),$ the third, fourth and fifth multi-time $h-$KCC-invariants are
given by%
\begin{equation*}
\overset{h}{R}\text{ \negthinspace \negthinspace }_{jk}^{i\alpha }=h^{\alpha
\mu }\mathfrak{R}_{pjk}^{i}x_{\mu }^{p},\qquad \overset{h}{B}\text{
\negthinspace \negthinspace }_{jk(l)}^{i\alpha (\beta )}=h^{\alpha \beta }%
\mathfrak{R}_{ljk}^{i},\qquad D_{(\alpha )\beta (j)(k)(l)}^{(i)\text{ \ }%
(\gamma )(\varepsilon )(\mu )}=0.
\end{equation*}
\end{example}

\begin{example}
For the first order PDE system (\ref{Jet_DS}) the third, fourth and fifth
multi-time $h-$KCC-invariants are zero.
\end{example}

\begin{theorem}[of characterization of the multi-time KCC-invariants]
Let $(T,h)$ be a Riemannian manifold$,$ where $m=\dim T\geq 3$. If the first
and the fifth KCC-invariants of the PDE system (\ref{SODE}) cancel on $%
J^{1}(T,M),$ then there exist on $J^{1}(T,M)$ some local functions $\Gamma
_{pq}^{i}(t,x),$ where $i,p,q=\overline{1,n},$ $n=\dim M,$ and $\mathtt{S}%
_{\alpha pq}^{i\nu }(t,x),$ $\alpha \neq \nu \in \{1,2,...,m\},$\ $i=%
\overline{1,n},$ $p\neq q\in \{1,2,...,n\},$ which have the properties 
\begin{equation*}
\Gamma _{pq}^{i}=\Gamma _{qp}^{i},\qquad \mathtt{S}_{\alpha pq}^{i\nu }+%
\mathtt{S}_{\alpha qp}^{i\nu }=0
\end{equation*}%
and (no sum by $\alpha $ or $\nu $) 
\begin{equation}
2\mathtt{S}_{\alpha pq}^{i\nu }=\sum_{\QATOP{\varepsilon {=1}}{\varepsilon {%
\neq \nu }}}^{m}\left[ h^{\varepsilon \varepsilon }\mathtt{S}_{\varepsilon
pq}^{i\nu }-h^{\nu \nu }\mathtt{S}_{\nu pq}^{i\varepsilon }\right]
h_{\varepsilon \alpha },  \label{**}
\end{equation}%
such that (no sum by $\alpha $ or $\beta $) 
\begin{equation}
F_{(\alpha )\beta }^{(i)}=\Gamma _{pq}^{i}x_{\alpha }^{p}x_{\beta
}^{q}-H_{\alpha \beta }^{\mu }x_{\mu }^{i}+2\delta _{\alpha \beta }\sum_{%
\QATOP{{\nu =1}}{{\nu \neq \alpha }}}^{m}\sum_{p\neq q\in \{1,2,...,n\}}%
\mathtt{S}_{\alpha pq}^{i\nu }x_{\alpha }^{p}x_{\nu }^{q},  \label{h-KCC=0}
\end{equation}%
where $\delta _{\alpha \beta }$ is the Kronecker symbol and $H_{\alpha \beta
}^{\gamma }$ are the Christoffel symbols of the Riemannian metric $h_{\alpha
\beta }(t)$.
\end{theorem}

\begin{proof}
By integration, the relations%
\begin{equation*}
D_{(\alpha )\beta (j)(k)(l)}^{(i)\text{ \ }(\gamma )(\varepsilon )(\sigma )}=%
\frac{\partial ^{3}F_{(\alpha )\beta }^{(i)}}{\partial x_{\gamma
}^{j}\partial x_{\varepsilon }^{k}\partial x_{\sigma }^{l}}=0,
\end{equation*}%
where $F_{(\alpha )\beta }^{(i)}=F_{(\beta )\alpha }^{(i)}$, subsequently
lead to 
\begin{eqnarray*}
\frac{\partial ^{2}F_{(\alpha )\beta }^{(i)}}{\partial x_{\gamma
}^{j}\partial x_{\varepsilon }^{k}} &=&2\Gamma _{(\alpha )\beta (j)(k)}^{(i)%
\text{ \ \ }(\gamma )(\varepsilon )}(t,x)\Rightarrow \\
&\Rightarrow &\frac{\partial F_{(\alpha )\beta }^{(i)}}{\partial x_{\gamma
}^{j}}=2\Gamma _{(\alpha )\beta (j)(q)}^{(i)\text{ \ \ }(\gamma )(\nu
)}x_{\nu }^{q}+\mathcal{U}_{(\alpha )\beta (j)}^{(i)\text{ \ \ }(\gamma
)}(t,x)\Rightarrow \\
&\Rightarrow &F_{(\alpha )\beta }^{(i)}=\Gamma _{(\alpha )\beta (p)(q)}^{(i)%
\text{ \ \ }(\mu )(\nu )}x_{\mu }^{p}x_{\nu }^{q}+\mathcal{U}_{(\alpha
)\beta (q)}^{(i)\text{ \ \ }(\nu )}x_{\nu }^{q}+\mathcal{V}_{(\alpha )\beta
}^{(i)}(t,x),
\end{eqnarray*}%
where%
\begin{equation}
\begin{array}{ll}
\Gamma _{(\alpha )\beta (j)(k)}^{(i)\text{ \ \ }(\gamma )(\varepsilon
)}=\Gamma _{(\beta )\alpha (j)(k)}^{(i)\text{ \ \ }(\gamma )(\varepsilon )},
& \Gamma _{(\alpha )\beta (j)(k)}^{(i)\text{ \ \ }(\gamma )(\varepsilon
)}=\Gamma _{(\alpha )\beta (k)(j)}^{(i)\text{ \ \ }(\varepsilon )(\gamma
)},\medskip \\ 
\mathcal{U}_{(\alpha )\beta (j)}^{(i)\text{ \ \ }(\gamma )}=\mathcal{U}%
_{(\beta )\alpha (j)}^{(i)\text{ \ \ }(\gamma )}, & \mathcal{V}_{(\alpha
)\beta }^{(i)}=\mathcal{V}_{(\beta )\alpha }^{(i)}.%
\end{array}
\label{symmetry-conditions}
\end{equation}

The equalities $\overset{h}{\varepsilon }$ $\!\!_{(\alpha )\beta }^{(i)}=0$
on $J^{1}(T,M)$ lead us to%
\begin{equation}
\begin{array}{lll}
\Gamma _{(\alpha )\beta (p)(q)}^{(i)\text{ \ \ }(\mu )(\nu )} & = & \dfrac{1%
}{2}\left[ \Gamma _{\text{ }(p)(q)}^{i(\eta )(\nu )}\delta _{\beta }^{\mu
}+\Gamma _{\text{ }(q)(p)}^{i(\eta )(\mu )}\delta _{\beta }^{\nu }\right]
h_{\eta \alpha },\medskip \\ 
\mathcal{U}_{(\alpha )\beta (q)}^{(i)\text{ \ \ }(\nu )} & = & \dfrac{1}{2}%
\mathcal{U}_{\text{ }(q)}^{i(\eta )}h_{\eta \alpha }\delta _{\beta }^{\nu }+%
\dfrac{1}{2}H^{\eta }h_{\eta \alpha }\delta _{\beta }^{\nu }\delta
_{q}^{i}-H_{\alpha \beta }^{\nu }\delta _{q}^{i},\medskip \\ 
\mathcal{V}_{(\alpha )\beta }^{(i)} & = & 0,%
\end{array}
\label{relations}
\end{equation}%
where%
\begin{equation*}
\Gamma _{\text{ }(p)(q)}^{i(\mu )(\nu )}=h^{\varepsilon \rho }\Gamma
_{(\varepsilon )\rho (p)(q)}^{(i)\text{ \ }(\mu )(\nu )}\text{ and }\mathcal{%
U}_{\text{ }(q)}^{i(\nu )}=h^{\varepsilon \rho }\mathcal{U}_{(\varepsilon
)\rho (q)}^{(i)\text{ \ }(\nu )}.
\end{equation*}%
Applying an $h-$trace in the second relation of (\ref{relations}), we deduce
that%
\begin{equation*}
\mathcal{U}_{(\alpha )\beta (q)}^{(i)\text{ \ \ }(\nu )}=-H_{\alpha \beta
}^{\nu }\delta _{q}^{i}.
\end{equation*}%
The first relation of (\ref{relations}) and the first symmetry properties of
(\ref{symmetry-conditions}) imply the following equalities:

\begin{enumerate}
\item for every $\alpha \neq \beta $ we have (no sum by $\alpha $ or $\beta $%
):

\begin{enumerate}
\item $\mu ,\nu \notin \{\alpha ,\beta \}\Rightarrow \Gamma _{(\alpha )\beta
(p)(q)}^{(i)\text{ \ \ }(\mu )(\nu )}=0;$

\item $\mu =\alpha ,$ $\nu =\alpha \Rightarrow \Gamma _{(\alpha )\beta
(p)(q)}^{(i)\text{ \ \ }(\alpha )(\alpha )}=0;$

\item $\mu =\alpha ,$ $\nu =\beta \Rightarrow \Gamma _{(\alpha )\beta
(p)(q)}^{(i)\text{ \ \ }(\alpha )(\beta )}=\dfrac{1}{2}\Gamma _{\text{ }%
(q)(p)}^{i(\eta )(\alpha )}h_{\eta \alpha }\overset{not}{=}\mathbb{S}%
_{\alpha pq}^{i}=\mathbb{T}_{\beta pq}^{i};$

\item $\mu =\beta ,$ $\nu =\alpha \Rightarrow \Gamma _{(\alpha )\beta
(p)(q)}^{(i)\text{ \ \ }(\beta )(\alpha )}=\dfrac{1}{2}\Gamma _{\text{ }%
(p)(q)}^{i(\eta )(\alpha )}h_{\eta \alpha }\overset{not}{=}\mathbb{T}%
_{\alpha pq}^{i}=\mathbb{S}_{\beta pq}^{i};$

\item $\mu =\beta ,$ $\nu =\beta \Rightarrow \Gamma _{(\alpha )\beta
(p)(q)}^{(i)\text{ \ \ }(\beta )(\beta )}=\Gamma _{(\beta )\alpha
(p)(q)}^{(i)\text{ \ \ }(\beta )(\beta )}=$

$=\dfrac{1}{2}\left[ \Gamma _{\text{ }(p)(q)}^{i(\eta )(\beta )}h_{\eta
\alpha }+\Gamma _{\text{ }(q)(p)}^{i(\eta )(\beta )}h_{\eta \alpha }\right] =%
\mathtt{S}_{\alpha pq}^{i\beta }+\mathtt{T}_{\alpha pq}^{i\beta }=0;$
\end{enumerate}

\item for every $\alpha =\beta \in \{1,2,...,m\}$ we obtain (no sum by $%
\alpha $):

\begin{enumerate}
\item $\mu \neq \alpha ,\nu \neq \alpha \Rightarrow \Gamma _{(\alpha )\alpha
(p)(q)}^{(i)\text{ \ \ }(\mu )(\nu )}=0;$

\item $\mu =\alpha ,\nu \neq \alpha \Rightarrow \Gamma _{(\alpha )\alpha
(p)(q)}^{(i)\text{ \ \ }(\alpha )(\nu )}=\dfrac{1}{2}\Gamma _{\text{ }%
(p)(q)}^{i(\eta )(\nu )}h_{\eta \alpha }\overset{not}{=}\mathtt{S}_{\alpha
pq}^{i\nu }=\mathtt{T}_{\alpha qp}^{i\nu };$

\item $\mu \neq \alpha ,\nu =\alpha \Rightarrow \Gamma _{(\alpha )\alpha
(p)(q)}^{(i)\text{ \ \ }(\mu )(\alpha )}=\dfrac{1}{2}\Gamma _{\text{ }%
(q)(p)}^{i(\eta )(\mu )}h_{\eta \alpha }\overset{not}{=}\mathtt{T}_{\alpha
pq}^{i\mu }=\mathtt{S}_{\alpha qp}^{i\mu };$

\item $\mu =\alpha ,$ $\nu =\alpha \Rightarrow \Gamma _{(\alpha )\alpha
(p)(q)}^{(i)\text{ \ \ }(\alpha )(\alpha )}=\Gamma _{(\alpha )\alpha
(q)(p)}^{(i)\text{ \ \ }(\alpha )(\alpha )}=$

$=\dfrac{1}{2}\left[ \Gamma _{\text{ }(p)(q)}^{i(\eta )(\alpha )}h_{\eta
\alpha }+\Gamma _{\text{ }(q)(p)}^{i(\eta )(\alpha )}h_{\eta \alpha }\right]
=\mathbb{T}_{\alpha pq}^{i}+\mathbb{S}_{\alpha pq}^{i};$
\end{enumerate}
\end{enumerate}

The first symmetry condition from (\ref{symmetry-conditions}), together with
1.(c) and 1.(d), give us ($m=\dim T\geq 3$)%
\begin{equation*}
\begin{array}{llllllllllllll}
\mathbb{S}_{1pq}^{i} & = & \mathbb{T}_{2pq}^{i} & = & \mathbb{T}_{3pq}^{i} & 
= & \mathbb{T}_{4pq}^{i} & = & \cdot & \cdot & = & \mathbb{T}_{mpq}^{i} & 
\overset{not}{=} & \dfrac{1}{2}\Gamma _{pq}^{i}\medskip \\ 
\mathbb{S}_{2pq}^{i} & = & \mathbb{T}_{1pq}^{i} & = & \mathbb{T}_{3pq}^{i} & 
= & \mathbb{T}_{4pq}^{i} & = & \cdot & \cdot & = & \mathbb{T}_{mpq}^{i} & 
\overset{not}{=} & \dfrac{1}{2}\Gamma _{pq}^{i}\medskip \\ 
\cdot & \cdot & \cdot & \cdot & \cdot & \cdot & \cdot & \cdot & \cdot & \cdot
& \cdot & \cdot & \cdot & \cdot \medskip \\ 
\cdot & \cdot & \cdot & \cdot & \cdot & \cdot & \cdot & \cdot & \cdot & \cdot
& \cdot & \cdot & \cdot & \cdot \medskip \\ 
\mathbb{S}_{mpq}^{i} & = & \mathbb{T}_{1pq}^{i} & = & \mathbb{T}_{2pq}^{i} & 
= & \mathbb{T}_{3pq}^{i} & = & \cdot & \cdot & = & \mathbb{T}_{(m-1)pq}^{i}
& \overset{not}{=} & \dfrac{1}{2}\Gamma _{pq}^{i}.%
\end{array}%
\end{equation*}%
Consequently, for every $\alpha \neq \beta \in \{1,2,...,m\}$ we have%
\begin{equation*}
\Gamma _{(\alpha )\beta (p)(q)}^{(i)\text{ \ \ }(\mu )(\nu )}=\dfrac{1}{2}%
\Gamma _{pq}^{i}\left[ \delta _{\alpha }^{\mu }\delta _{\beta }^{\nu
}+\delta _{\alpha }^{\nu }\delta _{\beta }^{\mu }\right]
\end{equation*}%
and for every $\alpha \in \{1,2,...,m\}$ we have%
\begin{equation*}
\Gamma _{(\alpha )\alpha (p)(q)}^{(i)\text{ \ \ }(\alpha )(\alpha )}=\Gamma
_{pq}^{i}=\Gamma _{(\alpha )\alpha (q)(p)}^{(i)\text{ \ \ }(\alpha )(\alpha
)}=\Gamma _{qp}^{i}.
\end{equation*}

Using now all the preceding properties, together with the equality 2.(b), we
find the equations (\ref{**}). Moreover, for every $\alpha \neq \nu \in
\{1,2,...,m\}$, it is obvious that we have%
\begin{equation*}
\mathtt{S}_{\alpha pq}^{i\nu }+\mathtt{S}_{\alpha qp}^{i\nu }=0\Rightarrow 
\mathtt{S}_{\alpha pp}^{i\nu }=0.
\end{equation*}

All the preceding situations can be briefly written in the general formula%
\begin{equation*}
\Gamma _{(\alpha )\beta (p)(q)}^{(i)\text{ \ \ }(\mu )(\nu )}=\dfrac{1}{2}%
\Gamma _{pq}^{i}\left[ \delta _{\alpha }^{\mu }\delta _{\beta }^{\nu
}+\delta _{\alpha }^{\nu }\delta _{\beta }^{\mu }\right] +\mathtt{S}_{\alpha
pq}^{i\nu }\delta _{\alpha \beta }\delta _{\alpha }^{\mu }\left[ 1-\delta
_{\alpha }^{\nu }\right] +\mathtt{S}_{\alpha qp}^{i\mu }\delta _{\alpha
\beta }\delta _{\alpha }^{\nu }\left[ 1-\delta _{\alpha }^{\mu }\right] .
\end{equation*}

In conclusion, we obtain the equalities (\ref{h-KCC=0}) on the 1-jet space $%
J^{1}(T,M)$.\medskip
\end{proof}

\textbf{Open problem.} If we fix the indices $i$ and $p\neq q$ in the set $%
\{1,2,...,n\}$, then we deduce that the system of equations (\ref{**}) is an
homogenous linear system of order $m(m-1)$. Consequently, it has at least
the zero solution. Because the coefficients of the system depend only by the
metric $h_{\alpha \beta }(t)$, there exist a temporal Riemannian metric $%
h_{\alpha \beta }(t)$ such that the system of equations (\ref{**}) to admit
only the banal solution?

{\small \setlength{\parskip}{0mm} }

{\small \noindent Authors' address: }

{\small \medskip \noindent Mircea NEAGU\newline
University Transilvania of Bra\c{s}ov, Faculty of Mathematics and
Informatics,\newline
Department of Algebra, Geometry and Differential Equations,\newline
B-dul Iuliu Maniu, Nr. 50, BV 500091, Bra\c{s}ov, Romania.\newline
E-mail: mircea.neagu@unitbv.ro\newline
Website: http://www.2collab.com/user:mirceaneagu }


\begin{thebibliography}{99}
\bibitem{1} {\small P.L. Antonelli, \textit{Equivalence Problem for Systems
of Second Order Ordinary Differential Equations}, Encyclopedia of
Mathematics, Kluwer Academic, Dordrecht, 2000. }

\bibitem{2} {\small V. Balan, M. Neagu, \textit{Jet geometrical extension of
the KCC-invariants}, http://arXiv.org/math.DG/0906.2903v2, (2009).}

\bibitem{3} {\small V. Balan, I.R. Nicola, \textit{Linear and structural
stability of a cell division process model}, Hindawi Publishing Corporation,
International Journal of Mathematics and Mathematical Sciences, vol. 2006,
1-15. }

\bibitem{4} {\small E. Cartan, \textit{Observations sur le m\'{e}moir pr\'{e}%
c\'{e}dent}, Mathematische Zeitschrift \textbf{37} (1933), no. \textbf{1},
619-622. }

\bibitem{5} {\small S.S. Chern, \textit{Sur la g\'{e}om\'{e}trie d'un syst%
\`{e}me d'equations differentialles du second ordre}, Bulletin des Sciences
Math\'{e}matiques \textbf{63} (1939), 206-212.}

\bibitem{6} {\small J. Eells, L. Lemaire, \textit{A report on harmonic maps}%
, Bulletin of London Mathematical Society \textbf{10} (1978), 1-68.}

\bibitem{7} {\small D.D. Kosambi, \textit{Parallelism and path-spaces},
Mathematische Zeitschrift \textbf{37} (1933), no. \textbf{1}, 608-618. }

\bibitem{8} {\small M. Neagu, \textit{Riemann-Lagrange Geometry on 1-Jet
Spaces}, Matrix Rom, Bucharest, 2005. }

\bibitem{9} {\small M. Neagu, C. Udri\c{s}te, A. Oan\u{a}, \textit{%
Multi-time dependent sprays and }$h$\textit{-traceless maps on }$J^{1}(T,M)$%
, Balkan Journal of Geometry and Its Applications, vol. \textbf{10} (2005),
no. \textbf{2}, 76-92. }

\bibitem{10} {\small V.S. Sab\u{a}u, \textit{Systems biology and deviation
curvature tensor}, Nonlinear Analysis. Real World Applications \textbf{6}
(2005), no. \textbf{3}, 563-587. }
\end{thebibliography}
\end{document}